\documentclass[sn-mathphys,Numbered]{sn-jnl}

\usepackage{graphicx}%
\usepackage{mathtools}
\usepackage{multirow}%
\usepackage{amsmath,amssymb,amsfonts}%
\usepackage{amsthm}%
\usepackage{mathrsfs}%
\usepackage[title]{appendix}%
\usepackage{xcolor}%
\usepackage{textcomp}%
\usepackage{manyfoot}%
\usepackage{booktabs}%
\usepackage{algorithm}%
\usepackage{algorithmicx}%
\usepackage{algpseudocode}%
\usepackage{listings}%
\usepackage[all]{xy}
\usepackage{tikz}

\theoremstyle{thmstyleone}%
\theoremstyle{thmstyletwo}%
\newtheorem{theorem}{Theorem}[section]
\newtheorem{proposition}{Proposition}[section]
\newtheorem{definition}{Definition}
\newtheorem{lemma}{Lemma}[section]
\newtheorem{corollary}{Corollary}
\newtheorem{problem}{Problem}
\newtheorem{remark}{Remark}

\newcommand\Z{{\mathbb{Z}}}

\newcommand\CC{{\mathbb{C}}}
\newcommand\HH{{\mathbb{H}}}

\newcommand{\be}{\begin{equation}}
	\newcommand{\ee}{\end{equation}}

\begin{document}

\title[Article Title]{Iterated residue, toric forms and Witten genus}


\author*[1]{\fnm{Hao} \sur{Li}}\email{14110840001@fudan.edu.cn}

\author[2]{\fnm{Fei} \sur{Han}}\email{mathanf@nus.edu.sg}
\equalcont{These authors contributed equally to this work.}

\author[3]{\fnm{Zhi} \sur{L\"u}}\email{zlu@fudan.edu.cn}
\equalcont{These authors contributed equally to this work.}

\affil*[1]{\orgdiv{College of Science}, \orgname{ University of Shanghai for Science and Technology},  \city{Shanghai},  \country{CHINA}}

\affil[2]{\orgdiv{School of Mathematics}, \orgname{National University of Singapore}, \country{Singapore}}

\affil[3]{\orgdiv{School of Mathematical Sciences}, \orgname{Fudan University},  \city{Shanghai},  \country{  CHINA}}


\abstract{ We introduce the notion of {\em iterated residue} to study generalized Bott manifolds. When applying the iterated residues to compute the Borisov--Gunnells toric form and the Witten genus of certain toric varieties as well as complete intersections, we obtain interesting vanishing results and some theta function identities, one of which is a twisted version of a classical Rogers--Ramanujan type formula.}

\keywords{generalized Bott manifold, Iterated residue, toric form, Witten genus}

\pacs[MSC Classification]{ 14M10, 52B20, 58J20}

\maketitle

\section{Introduction}\label{sec1}

The residue of a meromorphic function $f$ at an isolated singularity $a$, often denoted by $\mathrm{Res}_a \ f $, is the coefficient $a_{-1}$ of its Laurent series expansion at $a$. According to the Cauchy integral theorem, we have
$$\mathrm{Res}_a \ f =\frac{1}{2\pi i}\oint_\gamma f(z)dz,$$
where $\gamma$ is a circle small enough around $a$ in a counterclockwise manner. This residue has an interesting simple topological application on complex projective spaces. Actually,  for any top cohomology class $g$  in $H^*(\mathbb{C}P^n;\mathbb{C})\cong \mathbb{C} [z]/\langle z^{n+1}\rangle$ with the generator $z$ of degree 2, we have
$$ \langle g, \, [\mathbb{C}P^n]\rangle=\mathrm{Res}_0 \ \frac{g}{z^{n+1}} =\frac{1}{2\pi i}\oint_\gamma \frac{g}{z^{n+1}}dz.$$

For the case of multiple complex variables, there exists a similar story. Let $U$ be the open ball $\{\textbf{z}\in \mathbb{C}^n:|\textbf{z}| <\epsilon \}$ and $f_1,\cdots,f_n\in \mathcal{O}(\overline{U})$ be holomorphic functions in a neighborhood of the closure $\overline{U}$. Assume $f_i$'s have the origin as the isolated common zero and their Jacobian determinant is nonzero at the origin. In this case, $f_i$'s are said to be non-degenerate here.
The residue  of  meromorphic $n$-form
$$\omega=\frac{g(\textbf{z})dz_1\wedge \cdots\wedge dz_n}{f_1(\textbf{z})\cdots f_n(\textbf{z})} \ \ \ (g\in \mathcal{O}(\overline{U}))$$
is defined (\cite{GH}) by
\begin{equation}\label{multiresidue} \mathrm{Res}_{\textbf{0}}\   \omega =\frac{1}{(2\pi i)^n}\int_\Gamma \omega, \end{equation}
where $\Gamma$ is the real $n$-cycle defined by $\Gamma=\{ \textbf{z}: |f_i(\textbf{z})|<\delta\}$. The residue  is often called ``Grothendieck residue", and  has nice topological applications. Let $X$ be an oriented closed manifold such that $H^*(X;\mathbb{C})\cong \mathbb{C}[z_1,\cdots,z_n]/\langle f_1h_1,\cdots, f_nh_n\rangle$ with $f_1, ..., f_n$ being non-degenerate. Then for any top cohomology class $g= g'\cdot h_1 \cdots h_n$ in $H^*(X;\mathbb{C})$, one has
\begin{equation*} \langle g, \, [X]\rangle=\mathrm{Res}_\textbf{0} \ \frac{g'(\textbf{z})dz_1\wedge \cdots\wedge dz_n}{f_1(\textbf{z})\cdots f_n(\textbf{z})}.
\end{equation*}
\begin{remark}
	It should be pointed out that there is a quite restriction on the cohomology ring $H^*(X;\mathbb{C})$ and top cohomology class $g$ when one uses the residue to calculate $\langle g, \, [X]\rangle$. Namely, 
	the Jacobian determinant of $f_1, ..., f_n$ is non zero.
	Typical examples are  complex projective spaces (or their product) and Grassmannian manifolds (where the ideal arises from the Landau-Ginzburg potential, c.f. \cite{W}).
\vskip.2cm
Let $X=\mathbb{C}P^n\times \mathbb{C}P^m$,  $H^*(X;\mathbb{C})\cong \mathbb{C}[z_1,z_2]/\langle z_1\cdot z_1^n, z_2\cdot z_2^m\rangle$, where $f_1=z_1, h_1=z_1^n, f_2=z_2, h_2=z_2^m$, $f_i$'s have the origin as the isolated common zero and their Jacobian determinant is nonzero at the origin.  Then for any top cohomology class $g=g'\cdot z_1^nz_2^m$, we have
\begin{equation*} \langle g, \, [X]\rangle=\mathrm{Res}_\textbf{0} \ \frac{g'dz_1\wedge dz_2}{z_1z_2}=g'.\end{equation*}
	\end{remark}

Nevertheless when the  $\{f_i\}$'s in ideal are {\bf degenerate}, namely their Jacobian determinant vanishes,
  as far as the authors know, there are no appropriate ``global'' residue forms as in (\ref{multiresidue}) for  $\omega$.
\vskip.2cm
A natural question arises.
\vskip.2cm
\begin{problem}
	For degenerate case, which kinds of oriented manifolds can still admit a residue expression of the evaluation of their  cohomology classes on their fundamental homology classes?
\end{problem}
\vskip.2cm
In this paper, we investigated on generalized Bott manifolds with respect to above problem.
\vskip.2cm
Generalized Bott manifolds  are the total spaces of iterated projective bundles as a tower (c.f.   Section \ref{Bott manifold})
$$B_n \overset{p_n}{\longrightarrow} B_{n-1}\overset{p_{n-1}}{\longrightarrow} \cdots \longrightarrow B_2\overset{p_2}{\longrightarrow} B_1 \overset{p_1} \longrightarrow \text{pt} $$
where $B_k=\CC P(\xi_k\oplus \underline{\CC})$ and $\xi_k$ is the Whitney sum of $n_k$ complex line bundles over $B_{k-1}$.
Let $-u_k$ be the first Chern class of the tautological line bundle over  $B_k$ and $\{ x_{k1},x_{k2},\cdots,x_{kn_k} \}$ be the Chern roots of $\xi_k$.
Then the cohomology ring of $B_n$ is $$H^*(B_n)\cong \Z[u_1,\cdots,u_n]/\langle f_i(u_1,\cdots,u_n):\ i=1,\cdots, n\rangle,$$  where $f_i(u_1,\cdots,u_n)=u_i\prod_{j=1}^{n_i}(u_i+x_{ij})$ with $\{x_{ij}\}$ being the formal Chern roots.

\vskip.2cm

We shall introduce the concept ``\textbf{iterated residue}" (Definition \ref{iterated residue})
$$\mathrm{Res}_0\left\{\cdots \mathrm{Res}_0\left\{\frac{g(\textbf{z})}{f_1(\textbf{z})\cdots f_n(\textbf{z})}  dz_1\right\}\cdots\right\} dz_n.$$
and prove that

\vskip.2cm

\begin{theorem}[Theorem \ref{main}]
	For any top cohomology class $g $, one has
	$$\langle g, \, [B_n]\rangle=\mathrm{Res}_0\left\{\cdots \mathrm{Res}_0\left\{\frac{g(\textbf{u})}{f_1(\textbf{u})\cdots f_n(\textbf{u})}  du_1\right\}\cdots\right\} du_n,$$
	where the order of $u_i$ coincides with their position in the generalized Bott tower.
\end{theorem}

\vskip .2cm
\begin{remark}
	``Iterated residue" is similar to iterated integral in multiple variable calculus. The ``iterated residue" coincides with usual residue when $\{f_i\}$'s are non-degenerate.
\vskip .2cm	
	When the iterated projective bundles are trivial, generalized Bott manifolds $B_n$ are simply products of projective spaces, and iterated residue exactly reduces to Grothendieck residue.
\end{remark}

\vskip .2cm

Bott manifolds and generalized Bott manifolds forms a class of smooth projective toric varieties and have played  an important role in many respects. 

\begin{enumerate}
	\item[$\bullet$] Generalized Bott manifolds can serve as an alternative multiplicative generator set for unitary bordism ring, as shown in  Buchstaber and Ray~\cite{BR},   L\"u and Panov \cite{LP};
	\item[$\bullet$] The bordism classes of bounded flag manifolds $BF_n$( which are Bott manifolds) form a Ray's basis of
	the free $\Omega_U$-module $U_*(\mathbb{C}P^\infty)$;
	\item[$\bullet$] Bott manifolds  provide us with a simple model in which Schubert calculus can be simplified \cite{Du}.
\end{enumerate}


For arbitrary genera of complete intersections in generalized Bott manifolds, we obtain the following explicit expression in terms of iterated residue.
\vskip .2cm

\begin{theorem}[Theorem \ref{main1}]
	Denote by $X$ the submanifold of $B_n$ which is Poincar\'e dual to $a_1u_1+\cdots+a_nu_n \in H^2(B_n;\Z)$.
	For any genus $\psi$ with the characteristic power series $Q(x)=x/f(x)$, we have
	\begin{align*}
\label{genus} &\psi(X)\\= & \ \mathrm{Res}_0 \left\{ \cdots \left\{ \mathrm{Res}_0 \frac{f(a_1u_1+\cdots+a_nu_n)}{f(u_1)^{n_1+1}}\cdot \prod_{i=1}^{n}\frac{1}{f(u_i)}\cdot\prod_{j=1}^{n_i} \frac{1}{f(u_i+x_{ij})} d u_1 \right\} \cdots  \right\} d u_n. 
\end{align*}
\end{theorem}

\begin{remark}
It is a pity that above beautiful expression can not be generalized to toric
varieties or quasi-toric manifolds. Since for arbitrary toric variety $V$, the product of its Chern roots can not kill the ideals in $H^*(V; \mathbb{Z})$, thus $\psi(X)$ consists of  the function  $f$ as well as the monomials of $u_i$, we can not take advantage of the property of f.
\end{remark}


\vskip .2cm

We will present two applications of  above formula in this paper. One is concerned with the Borisov-Gunnells toric form \cite{BG} and the other one is concerned with the Witten genus \cite{W}.

\vskip .2cm

Toric varieties are very important objects in both algebraic geometry and geometric topology. Generalized Bott manifolds are examples of toric varieties.

\vskip .2cm

In \cite{BL}, Borisov and Libgober gave a combinatorial formula for the elliptic genus of a toric variety in terms of the associated fan, and  showed  that elliptic genera of a Calabi-Yau hypersurface in a toric variety and its mirror coincide up to sign.

\vskip .2cm

In \cite{BG}, given a toric variety $X$ and  a ``degree function" $\mathrm{deg}$,  Borisov and Gunnells introduced a function $f_{N, \mathrm{deg}}$ on the upper half plane $\HH$, called {\em toric form}, and proved that it is a kind of Euler characteristic number by Hirzebruch-Riemann-Roch theorem, and it is also a modular form when $\mathrm{deg}$ satisfies certain natural condition. In particular, Borisov and Libgober calculated toric form on $\mathbb{C}P^2$ \cite{BL} and got the classical Rogers--Ramanujan type formula (c.f. \cite{ASW, Ba})
\begin{equation}\label{RR} \sum_{m,n\geq 1}\frac{q^{m+n}}{(1+q^m)(1+q^n)(1+q^{m+n})}=\sum_{r\geq 1}\sigma_1(r)q^{2r}.
\end{equation}

We first apply our iterated residue formula to calculate the toric forms of low dimensional generalized Bott manifolds, and obtain a  family of new  theta function identities. However, the calculation would become much more complicated for higher dimensional generalized Bott manifolds. Let  $\theta(z, \tau)$ denote the classical Jacobi theta function (c.f. \cite{Ch85}) $$\theta(z,\tau)=2 q^{1/8}\sin(\pi z)\prod_{j=1}^\infty [(1-q^j)(1-e^{2\pi i  z}q^j)(1-e^{-2\pi i  z}q^j)],$$
where \  $q=\exp^{2\pi i\tau }$.
\vskip .2cm

For a $\mathbb{C}P^2$ bundle over $\mathbb{C}P^2$, we have  theta function identity:

\vskip .2cm
\begin{proposition}[\protect Proposition \ref{second}]
	When $j , k \in \mathbb{Z}$,
	one has		
\begin{align}\label{twistedRR}
			&(2\pi i)^4 \sum_{a,b,c,d\in \mathbb{Z}} \frac{2(1+q^{jc+kd})}{(1+q^a)(1+q^b)(1+q^{c})(1+q^{d})(1+q^{jc+kd-a-b})(1+q^{-c-d}) }\nonumber \\
			=\ & \Big(\frac{3\theta''(\frac{1}{2},\tau)}{2\theta(\frac{1}{2},\tau)}-\frac{\theta^{(3)}(0,\tau)}{2\theta'(0,\tau)}\Big)^2+(j^2+k^2-jk)
			\Big(\frac{\theta''^2(\frac{1}{2},\tau)}{4\theta^2(\frac{1}{2},\tau)}	+\frac{5\theta^{(4)}(\frac{1}{2},\tau)}{24\theta(\frac{1}{2},\tau)}\nonumber \\
& -\frac{7\theta^{(3)}(0,\tau)\theta''(\frac{1}{2},\tau)}{12\theta'(0,\tau)\theta(\frac{1}{2},\tau)}			-\frac{\theta^{(5)}(0,\tau)}{24\theta'(0,\tau)}+\frac{(\theta^{(3)}(0,\tau))^2}{6\theta'^2(0,\tau)}\Big)
	 	\end{align}
where  $q=\exp^{2\pi i\tau}$.
\end{proposition}

\vskip .2cm

Our formula (\ref{twistedRR}) can be viewed as a twisted version of the classical Rogers--Ramanujan type formula (\ref{RR}), in the special case when $j=k=0$, this formula reduces to the square of formula (\ref{RR}), see Section \ref{classic examples} for details.

\vskip.2cm
For Hirzebruch surface  $F_k=\mathbb{C} P(\underline{\mathbb{C}}\oplus\mathcal{O}(k))$, we have  theta function identity:

\vskip .2cm

\begin{proposition}[\protect Proposition \ref{Hir}] When $\alpha_1, \alpha_2, \alpha_3, \alpha_4 \notin \mathbb{Z}$ are complex numbers, one has
	\begin{align*}
					&\sum_{a,b\in \mathbb{Z}} \frac{(1-e^{2\pi i (\alpha_3+\alpha_4)})(1-e^{2\pi i (\alpha_1+\alpha_2)}q^{ka})}{(1-e^{2\pi i \alpha_4}q^a)(1-e^{2\pi i \alpha_1}q^b)(1-e^{2\pi i \alpha_3}q^{-a})(1-e^{2\pi i \alpha_2}q^{ka-b}) }\\
			=&\frac{1}{(2\pi i)^2} \Big\{\Big(\frac{\theta'(-\alpha_3,\tau)}{\theta(-\alpha_3,\tau)}+ \frac{\theta'(-\alpha_4,\tau)}{\theta(-\alpha_4,\tau)}\Big)\Big(\frac{\theta'(-\alpha_1,\tau)}{\theta(-\alpha_1,\tau)}+ \frac{\theta'(-\alpha_2,\tau)}{\theta(-\alpha_2,\tau)}\Big)\\
& +\frac{k}{2}\Big(\frac{\theta''(-\alpha_3,\tau)}{\theta(-\alpha_3,\tau)}- \frac{\theta''(-\alpha_4,\tau)}{\theta(-\alpha_4,\tau)} \Big)\Big\}
	 \end{align*}
where  $q=\exp^{2\pi i\tau}$.
\end{proposition}

\vskip .2cm

In the special case when $k=0$, this formula implies  the following simple  Ramanujan type identity
$$\sum_{n=1}^{\infty}\frac{n\cdot q^n}{1-q^n}=\sum_{n=1}^{\infty}\frac{q^n}{(1-q^n)^2}.$$

We are also able to obtain the following vanishing result for some high dimensional generalized Bott manifolds.

\vskip .2cm

\begin{theorem} [\protect Theorem \ref{vanish toric}]
Let $V$ be the toric variety $\CC P(\eta^{\otimes i_1}\oplus \eta^{\otimes i_2}\oplus \eta^{\otimes i_3}\oplus \underline{\mathbb{C}})$ over $\CC P^{n_1}$, where $\eta$ denote the tautological bundle of $\CC P^{n_1}$.
	Assume that $(i_1, i_2, i_3)$ are coprime such that $ \ (i_1,\ i_1-i_2,\ i_1-i_3), \ (i_2,\ i_2-i_1,\ i_2-i_3), \ (i_3,\ i_3-i_1,i\ _3-i_2)$ are also coprime respectively, $\sum_{j=0}^{3} \alpha_{n_1+2+j}\in \mathbb{Z}$, and $\sum_{i=1}^{n_1+1} \alpha_i- \sum_{j=1}^{3}i_j \alpha_{n_1+2+j}\in \mathbb{Z}$, then the toric form $f_{N,\deg}(q)=0$, the definition of toric form can be found in Section~\ref{toric forms}.
\end{theorem}

\vskip .2cm

\begin{remark}
	The coprime condition seems to be a bit strict. We will provide an example to explain it. Let $(i_1, i_2, i_3)=(1,3,4)$, we can check that $(1,3,4),\ (1,2,3),\ (3,2,1)$ and $  (4,3,1)$ are coprime respectively. It is possible to transfer iterated residue into residues at simple poles by residue theorem when coprime condition holds.
\end{remark}
\vskip .2cm

Another application is about Witten genus, which plays an important role in index theory. Let $M$ be a $4m$ dimensional compact oriented smooth manifold. Let $\{\pm
2\pi \sqrt{-1}z_{j},1\leq j\leq 2m\}$ denote the formal Chern roots of $T_{%
	\mathbb{C}}M $, the complexification of the tangent vector bundle $TM$ of $M$%
. Then the famous Witten genus of $M$ can be written as (c.f. \cite{Liu})%
\begin{equation*}
	\varphi_W(M)=\left\langle \left( \prod_{j=1}^{2m}z_{j}\frac{\theta ^{\prime }(0,\tau
		)}{\theta (z_{j},\tau )}\right) ,[M]\right\rangle \in \mathbb{Q}[[q]],
\end{equation*}%
with $\tau \in \mathbb{H}$, the upper half-plane, and $q=e^{\pi \sqrt{-1}%
	\tau }$. The Witten genus was first introduced in \cite{W} and can be viewed as the
loop space analogue of the $\widehat{A} $-genus. It can be expressed as a $q$%
-deformed $\widehat{A}$-genus as
\begin{equation*}
	\varphi_W(M)=\left\langle \widehat{A}(TM)\mathrm{ch}\left( \Theta \left( T_{\mathbb{C%
	}}M\right) \right) ,[M]\right\rangle ,
\end{equation*}%
where
\begin{equation*}
	\Theta (T_{\mathbb{C}}M)=\overset{\infty }{\underset{n=1}{\otimes }}%
	S_{q^{2n}}(\widetilde{T_{\mathbb{C}}M}),\ \ {\rm with}\ \
	\widetilde{T_{\mathbb{C}}M}=T_{\mathbb{C}}M-{\mathbb C}^{4m},
\end{equation*}%
is the Witten bundle defined in \cite{W}. When the manifold $M$ is Spin, according to the
Atiyah-Singer index theorem \cite{AS},  the Witten genus can be expressed analytically as the index of twisted Dirac operators, $\varphi_W(M)=\mathrm{ind}(D\otimes \Theta \left( T_{\mathbb{C%
}}M\right))\in \mathbb{Z}[[q]]$, where $D$ is the Atiyah-Singer Spin Dirac operator on $M$ (c.f. \cite{HBJ94}). Moreover, if $M$ is String,  i.e. $\frac{1}{2}p_{1}(TM)=0$, or even weaker, if $M$ is Spin and the
first rational Pontryagin class of $M$ vanishes, then $\varphi_W(M)$ is a
modular form of weight $2k$ over $SL(2,\mathbb{Z})$ with integral
Fourier expansion (c.f. \cite{Za}). The homotopy theoretical
refinements of the Witten genus on String manifolds lead to the
theory of tmf ({\em topological modular form}) developed by Hopkins and
Miller \cite {Hop}. The String condition is the orientablity condition for this generalized cohomology theory.

\vskip .2cm

The Lichnerowicz theorem\cite{Lic} asserts that if a closed spin manifold carries a Riemannian metric of positive scalar curvature, then its $\widehat{A}$ genus vanishes.  Along this line, Stolz conjectured \cite{St1} that the Witten genus of a string manifold carrying positive Ricci curvature metric vanishes. There are two kinds of vanishing results to support the Stolz's conjecture. One is the theorem asserting that every string manifold carrying a nontrivial action of a semi-simple Lie group $G$ has vanishing Witten genus (\cite{De}, \cite{Liu2}). The other is the Landweber-Stong vanishing theorem (c.f. page 89 in \cite{HBJ94}) asserting that a string complete intersection in complex projective spaces has vanishing Witten genus. The proof uses the calculation of residues.  The Landweber-Stong type vanishing results were also obtained  for the following objects: string complete intersections in products of complex projective spaces (\cite{CH}, \cite{CHZ}), string complete intersections in products of Grassmannians and flag manifolds (\cite{ZhZh,Zh}).

\vskip .2cm

In this paper, by applying the iterated residue, we  can prove a Landweber-Stong type vanishing theorem for the Witten genus of complete intersections in two-staged generalized Bott manifolds.



\vskip .2cm

\begin{theorem}[Theorem \ref{technique}]\label{1.2.}
 Suppose that $H^{\bf I}_{n_1, 3}(d_1, d_2; d_3,d_4)$ is a string complete intersection. If ${\bf I}=(i,\  j,\  k)$ are coprime such that $ \ (i,\ i-j,\ i-k), \ (j,\ j-i,\ j-k), \ (k,\ k-i,\ k-j)$ are also coprime respectively, then the Witten genus
	$$\varphi_W(H^{\bf I}_{n_1, 3}(d_1, d_2; d_3,d_4))=0.$$
\end{theorem}

\vskip .2cm

Furthermore, it is also possible to calculate the mod 2 Witten genus of generalized Bott manifolds by making use of Rokhlin congruence formula in \cite{Zh1, Zh2}. We will study this in a forthcoming project.

\vskip .2cm

The paper is organised as follow. In Section \ref{Bott manifold}, we introduce the concept of ``iterated residue" and give an explicit expression of the genus of complete intersections in generalized Bott manifolds. In Section \ref{toric forms}, we apply ``iterated residue" to toric forms and get interesting theta function identities.   In Section \ref{Witten genus}, we will use ``iterated residue" to discuss when  Witten genus of complete intersections in generalized Bott manifolds vanishes and give the proof of  Theorem \ref{1.2.}.

\section{Iterated residue   in generalized Bott manifolds }\label{Bott manifold}

The Bott tower consists of  a family of complex manifolds obtained as the total spaces of iterated bundles over $\CC P^1$ with fibre $\CC P^1$. A Bott tower can be obtained as a limit of the complex structure on  Bott-Samelson varieties \cite{BS}.  Grossberg and Karshon \cite{GK} showed that  the Bott towers form an important family of smooth projective toric varieties. The generalized Bott tower was further introduced by Choi, Masuda and Suh
\cite{CMS} in the study of cohomological rigidity problem.

\vskip .2cm

\begin{definition}(c.f. \cite{CMS})
	A \textbf{generalized Bott tower} of height $n$ is a tower of projective bundles
	$$B_n \overset{p_n}{\longrightarrow} B_{n-1}\overset{p_{n-1}}{\longrightarrow} \cdots \longrightarrow B_2\overset{p_2}{\longrightarrow} B_1\longrightarrow \text{pt}, $$
	where $B_1=\CC P^{n_1}$ and each $B_k$ is the complex projectivization of  sum of $n_k$ complex line bundles and one trivial line bundle over $B_{k-1}$. The fibre of the bundle $p_k: B_k\longrightarrow B_{k-1}$ is $\CC P^{n_k}$.

\vskip .2cm
The last stage $B_n$ in a generalized Bott tower is called a \textbf{generalized Bott manifold}.
\end{definition}

\vskip .2cm

The topology information of generalized Bott manifolds is very clear. Its cohomology ring relies on the following  result:

\vskip .2cm

\begin{theorem}(c.f. \cite[Chapter V]{St})
	Let $\CC P(\xi)$ be the complex projectivization of $\xi$, where 
$\xi$ is a complex $n$-dimensional vector bundle over a finite cell complex $X$.  
 Then $H^*(\CC P(\xi))$ is the quotient of polynomial ring $H^*(X)[u]$ by single relation
	$$u^n+c_1(\xi)u^{n-1}+\cdots+c_n(\xi)=0$$
where  $-u\in H^2(\CC P(\xi))$ is the first Chern class of the tautological line bundle over $\CC P(\xi)$.
\end{theorem}

\vskip .2cm

\begin{corollary}
	Let $-u_k$ be the first Chern class of tautological line bundle $\eta_k$ over 
the generalized Bott manifold
$B_k=\CC P(\xi_k\oplus \underline{\CC})$.  Then
	$$H^*(B_n)\cong \Z[u_1,\cdots,u_n]/\langle f_i(u_1,\cdots,u_n):\ i=1,\cdots, n\rangle,$$
	where $f_i(u_1,\cdots,u_n)=u_i\prod_{j=1}^{n_i}(u_i+x_{ij})$, and $\{ x_{k1},x_{k2},\cdots,x_{kn_k}\}$ consists of the formal Chern roots of $\xi_k$.
\end{corollary}

\vskip .2cm
We note that the  tangential bundle of $B_{k+1}$ (over $B_k$) is clear, i.e.,
$$T B_{k+1}\oplus \underline{\CC}\cong p^*(T B_k)\oplus (\bar{\eta}_k\otimes p^*(\xi_k\oplus \underline{\CC})).$$

The generalized Bott manifolds provide a rich family of toric varieties on which we can calculate all kinds of characteristic numbers and study their geometry and topology.
\vskip.2cm
In \cite{LHLL}, the authors introduced the concept of {\em twisted Milnor hypersurfaces}, which is a generalization of Milnor hypersurfaces. 
Consider the projective bundle over $\mathbb{C} P^{n_{1}}$ with fiber $\mathbb{C} P^{n_{2}}$, i.e.
$$V=\mathbb{C} P(\eta^{\otimes i_{1}}\oplus\cdots\oplus\eta^{\otimes i_{n_{2}}}\oplus\underline{\mathbb{C}})\rightarrow \mathbb{C} P^{n_1},$$
where $\eta$ is the tautological line bundle over $\mathbb{C} P^{n_{1}}$ and $\underline{\mathbb{C}}$ is the trivial line bundle.

\vskip.2cm

Let $\gamma$ be the vertical tautological line bundle over $V$, $u=c_1(\overline{\eta}), v=c_1(\overline{\gamma})\in H^2(V; \Z)$. Denote $\mathbf{I}=(i_{1},\cdots ,i_{n_{2}})$ be the index.

\vskip.2cm

\begin{definition}  We call the smooth hypersurface  Poincar\'e dual to $d_1u+d_2v$ in $V$  {\bf twisted Milnor hypersurface}, denoted by $H_{n_1,n_2}^{\bf I}(d_1,d_2)$.
\end{definition}

\vskip.2cm

\begin{remark}
	We can also consider hypersurfaces or complete intersections on generalized Bott manifolds, such as $H_{n_1,n_2}^{\bf I}(d_1,d_2; d_3, d_4)$ denote the complete intersection Poincar\'e dual to $(d_1u+d_2v)\cdot (d_3u+d_4v)$ in $V$.
	Although generalized Bott manifolds are toric varieties,  their  topological complete intersections are  not necessarily algebraic (\cite{LHLL}).
\end{remark}

\vskip.2cm


Denote $X$ be the submanifold of $B_n$ Poincar\'e dual to $a_1u_1+\cdots+a_nu_n \in H^2(B_n;\Z)$.
Suppose $\nu$ denotes the normal bundle of inclusion $i: X\hookrightarrow B_n$. Then \\ $c_1(\nu)=i^*(a_1u_1+\cdots+a_nu_n)$, while  $i^*(T B_n)\cong T X \oplus \nu$.

\vskip.2cm
For any genus $\psi$ with the characteristic power series $Q(x)=x/f(x)$, we have
\begin{equation}\label{formula} \psi(X)=\big\langle \big(\frac{u_1}{f(u_1)}\big)^{n_1+1}\cdot f(a_1u_1+\cdots+a_nu_n)\cdot \prod_{i=1}^{n}\frac{u_i}{f(u_i)}\cdot\prod_{j=1}^{n_i} \frac{u_i+x_{ij}}{f(u_i+x_{ij})}, [B_n]\big\rangle.
\end{equation}

\vskip.2cm

If $B_n$ is a product of projective spaces,  $\psi(X)$ can be simplified into a very neat  expression. As to general $B_n$, the ideal $f_i $'s in $H^*(B_n; \Z)$ will be the biggest obstacle in calculating genus,  so we have to cope with the relation between $u_i$'s carefully.

\vskip.2cm

Motivated by the idea of Witten \cite{W}, we try to reduce $\psi(X)$ to  Grothendieck residue. Recall the global residue theorem

\vskip.2cm

\begin{theorem}$(${\bf Global Residue Theorem \cite[Chapter 5]{GH}}$)$
	Let $M$ be a compact complex $n$-manifold, $D=D_1+\cdots+D_n$ be a divisor on $M$ such that $D_1,\cdots,D_n$ are effective divisors on $M$ and the intersection $D_1\cap\cdots\cap D_n$ is discrete, hence finite-set of points in $M$. Then for any $\omega\in H^0(M,\Omega^n(D))$,
	$$\sum_{P\in \{D_1\cap\cdots\cap D_n \}} \mathrm{Res}_P\ \omega=0.$$
\end{theorem}

Global Residue Theorem does not apply to generalized Bott manifolds, since  the ideals $f_i$'s are not necessarily non-degenerate, thus Grothendieck residue at point $(0, \cdots, 0)$ can  not be well defined.

\vskip.2cm

To overcome this obstacle, we introduce the concept ``iterated residue".

\vskip.2cm

\begin{definition}\label{iterated residue}
	Let $U$ be the ball $\{\textbf{z}\in \mathbb{C}^n:|\textbf{z}| <\epsilon \}$ and $f_1,\cdots,f_m\in \mathcal{O}(\overline{U})$ be functions holomorphic in a neighborhood of the closure $\overline{U}$ of $U$.
	Then ``\textbf{iterated residue}"  of a meromorphic $n$-form
	$$\omega=\frac{g(\textbf{z})dz_1\wedge \cdots\wedge dz_n}{f_1(\textbf{z})\cdots f_m(\textbf{z})} \ \ \ (g\in \mathcal{O}(\overline{U}))$$
	is defined   by
	$$\mathrm{Res}_0\left\{\cdots \mathrm{Res}_0\left\{\frac{g(\textbf{z})}{f_1(\textbf{z})\cdots f_m(\textbf{z})}  dz_1\right\}\cdots\right\} dz_n. $$
\end{definition}

\vskip.2cm

\begin{remark}
``Iterated residue" is well-defined since the residue of a meromorphic function in
one complex variable is always well-defined.
\vskip.1cm

``Iterated residue" is similar to iterated integral in multiple calculus, the order of $u_i$ coincides with their position in the generalized Bott tower, ``iterated residue" coincides with usual residue when $f_i(z)$'s are non-degenerate.
\vskip.1cm
	
``Iterated residue" admits a topological interpretation via characteristics of generalized Bott manifolds.
\end{remark}

\begin{theorem}\label{main}
	For any $F\in H^{n_1+\cdots+n_n}(B_n;\Z)$, we have
	$$\langle \ F ,\ [B_n]\ \rangle=\mathrm{Res}_0 \left\{ \mathrm{Res}_0 \cdots \left\{ \mathrm{Res}_0 \frac{F}{\prod_{j=1}^{n}f_j(u_1,\cdots,u_n)}d u_1 \right\} \cdots d u_{n-1} \right\} d u_n,$$
	where the order of $u_i$ coincides with their position in the generalized Bott tower.
\end{theorem}
\begin{proof}
	We see that  $\langle F, [B_n]\rangle$ and $$\mathrm{Res}_0 \left\{ \mathrm{Res}_0 \cdots \left\{ \mathrm{Res}_0 \frac{F}{\prod_{j=1}^{n}f_j(u_1,\cdots,u_n)}d u_1 \right\} \cdots d u_{n-1} \right\} d u_n$$ are both linear forms, and  $H^{n_1+\cdots+n_n}(B_n;\Z)\cong \Z \langle u_1^{n_1}u_2^{n_2}\cdots u_n^{n_n} \rangle$.

\vskip .2cm 
	First let us look at the case $F=u_1^{n_1}u_2^{n_2}\cdots u_n^{n_n}$. 
	Obviously, $\langle u_1^{n_1}u_2^{n_2}\cdots u_n^{n_n}, [B_n]\rangle=1$, while
	since $u_1^{n_1+1}=0$ and
	$$ \mathrm{Res}_0\frac{ u_1^{n_1}u_2^{n_2}\cdots u_n^{n_n}}{\prod_{j=1}^{n}f_j(u_1,\cdots,u_n)}du_1
	=  \frac{u_2^{n_2}\cdots u_n^{n_n}}{\prod_{j=2}^{n}f_j(u_2,\cdots,u_n)}$$
so	
\begin{align*}
		& \mathrm{Res}_0 \left\{ \mathrm{Res}_0 \cdots \left\{ \mathrm{Res}_0 \frac{ u_1^{n_1}u_2^{n_2}\cdots u_n^{n_n}}{\prod_{j=1}^{n}f_j(u_1,\cdots,u_n)}d u_1 \right\} \cdots d u_{n-1} \right\} d u_n \\
		= & \mathrm{Res}_0 \left\{  \cdots \left\{ \mathrm{Res}_0 \frac{ u_2^{n_2}\cdots u_n^{n_n}}{\prod_{j=2}^{n}f_j(u_2,\cdots,u_n)}d u_2 \right\} \cdots  \right\} d u_n.\\
	\end{align*}
Continuing this procedure for $du_2, ..., du_n$, we obtain that
	\begin{align*}
		& \mathrm{Res}_0 \left\{ \mathrm{Res}_0 \cdots \left\{ \mathrm{Res}_0 \frac{ u_1^{n_1}u_2^{n_2}\cdots u_n^{n_n}}{\prod_{j=1}^{n}f_j(u_1,\cdots,u_n)}d u_1 \right\} \cdots d u_{n-1} \right\} d u_n =1
	\end{align*}
as desired. 
	\vskip.2cm
In general, $F$ may be a cohomology class  different from  $u_1^{n_1}u_2^{n_2}\cdots u_n^{n_n}$ in $H^{n_1+\cdots+n_n}(B_n;\Z)$. More precisely,  $F$ will be a multiple $a_F$ of $u_1^{n_1}u_2^{n_2}\cdots u_n^{n_n}$ by moduloing the ideal $f_i$'s, i.e.
$$F=a_F\cdot u_1^{n_1}u_2^{n_2}\cdots u_n^{n_n}+ \sum_{j=1}^{n}f_j\cdot g_j,$$
where $\langle f_j\rangle$ is an ideal in $H^{n_1+\cdots+n_n}(B_n;\Z)$, and $g_j$ is a polynomial of degree $n_1+\cdots+n_n-n_j$.

\vskip.2cm
We can compute directly that
$$ \mathrm{Res}_0 \frac{f_1\cdot g_1}{\prod_{j=1}^{n}f_j(u_1,\cdots,u_n)}d u_1 = \mathrm{Res}_0 \frac{ g_1}{\prod_{j=2}^{n}f_j(u_1,\cdots,u_n)}d u_1 =0$$
and 
\begin{align*}
   & \mathrm{Res}_0  \left\{ \mathrm{Res}_0 \frac{f_2\cdot g_2}{\prod_{j=1}^{n}f_j(u_1,\cdots,u_n)}d u_1 \right\}  d u_2  \\
  = &  \mathrm{Res}_0  \left\{ \mathrm{Res}_0 \frac{ g_2}{u_1^{n+1}\cdot\prod_{j=3}^{n}f_j(u_1,\cdots,u_n)}d u_1 \right\}  d u_2 \\
  = &  \mathrm{Res}_0 \frac{ g_2}{h_2(u_2,\cdots,u_n)}d u_2=0
\end{align*}
where $h_2(u_2,\cdots,u_n)$ has no pole at $u_2=0$.
 In the same argument way as above, we can conclude that $\mathrm{Res}_0 \left\{ \mathrm{Res}_0 \cdots \left\{ \mathrm{Res}_0 \frac{ f_k\cdot g_k}{\prod_{j=1}^{n}f_j(u_1,\cdots,u_n)}d u_1 \right\} \cdots d u_{n-1} \right\} d u_n =0,\ k=1,\cdots n.$
 Thus $$\langle \ F ,\ [B_n]\ \rangle=\mathrm{Res}_0 \left\{ \mathrm{Res}_0 \cdots \left\{ \mathrm{Res}_0 \frac{F}{\prod_{j=1}^{n}f_j(u_1,\cdots,u_n)}d u_1 \right\} \cdots d u_{n-1} \right\} d u_n=a_F,$$
 which implies our algorithm is well defined.
\end{proof}
\vskip.3cm
For a generalized Bott manifold $B_n$
$$B_n \overset{p_n}{\longrightarrow} B_{n-1}\overset{p_{n-1}}{\longrightarrow} \cdots \longrightarrow B_2\overset{p_2}{\longrightarrow} B_1\longrightarrow \text{pt}. $$

Let $-u_k$ be the first Chern class of the tautological line bundle over  $B_k=\CC P(\xi_k\oplus \underline{\CC})$ over $B_{k-1}$, the Chern roots of $\xi_k$ be $\{ x_{k1},x_{k2},\cdots,x_{kn_k} \}$. Denote $X$ be the submanifold of $B_n$ Poincar\'e dual to $a_1u_1+\cdots+a_nu_n \in H^2(B_n;\Z)$.

\vskip .2cm

Applying Theorem \ref{main} to formula (\ref{formula}), we conclude  that

\vskip .2cm
\begin{theorem}\label{main1}
	For any genus $\psi$ with the characteristic power series $Q(x)=x/f(x)$, we have
	\begin{align}\label{genus} &\psi(X)
\\
=& \ \mathrm{Res}_0 \left\{ \cdots \left\{ \mathrm{Res}_0 \frac{f(a_1u_1+\cdots+a_nu_n)}{f(u_1)^{n_1+1}}\cdot \prod_{i=1}^{n}\frac{1}{f(u_i)}\cdot\prod_{j=1}^{n_i} \frac{1}{f(u_i+x_{ij})} d u_1 \right\} \cdots  \right\} d u_n. \nonumber
\end{align}
	
\end{theorem}

Formula (\ref{genus}) kills the  ideals in $H^*(B_n; \mathbb{Z})$ and only $f$ is left. If power series $f$ admits good property, it is possible to give explicit formula of $\psi(X)$.
\vskip.2cm

\section{Borisov-Gunnellls Toric forms: theta functions identities and vanishing results}\label{toric forms}
Toric variety is a very important object in both algebraic geometry and geometric topology. Borisov and Libgober\cite{BL}  gave the explicit formula of elliptic genus of  a toric variety by its combinatorial data, and showed that elliptic genera of a Calabi-Yau hypersurface in a toric variety and its mirror coincide up to sign.
\vskip.2cm
In \cite{BG}, Borisov and Gunnells defined algebraic \emph{toric form} which is motivated from the expression of normalized elliptic genus of toric variety\cite{BL}.
\vskip.2cm
Let $N\in \mathbb{R}^r$ be a lattice, $M$ be its dual lattice. For a complete rational polyhedral fan $\Sigma\subset N\otimes \mathbb{R}$. A \emph{degree function} $\deg : N\longrightarrow \mathbb{C}$ is a piecewise linear function on the cones of $\Sigma$.
\vskip.2cm
For every cone $C\in \Sigma$, they defined a map
$$f_C: \mathbb{H}\times M_\mathbb{C}\longrightarrow \mathbb{C}$$
as follows. Write $q=e^{2\pi i \tau}$, $\tau\in \mathbb{H}$, the upper halfplane. If $m\in M_\mathbb{C}$ satisfies
$$m\cdot (C\setminus \{0\})>0,$$
for all $\tau$ with sufficiently large imaginary part, they set
$$f_C(q,m)\coloneqq \sum_{n\in C \cap N}  q^{m\cdot n}e^{2\pi i \deg(n)}.$$

\vskip.2cm
The  \emph{toric form} associated to $(N,\deg)$ is the function $f_{N,\deg}: \mathbb{H}\longrightarrow \mathbb{C}$  defined by
$$f_{N,\deg}(q)\coloneqq \sum_{m\in M} (   \sum_{C\in \Sigma}(-1)^{codim \ C} a.c.(  \sum_{n\in C} q^{m\cdot n}e^{2\pi i \deg(n)} )  ),$$
here  ``$a.c.$'' denotes analytic continuation of $f_C$.

\vskip.2cm

The definition of $f_{N,\deg}$ is well-defined. More precisely, there exists $\epsilon>0$ such that the sum over $M$ converges absolutely and uniformly for all $|q|<\epsilon$.

\vskip.2cm
\begin{remark}
	Suppose that $\deg(d)=1/2$ for all generators $d$ of one-dimensional cones of $\Sigma$, and toric variety $X$ associated to $\Sigma$ is nonsingular. Then the function $f_{N,\deg}(q)$ is the normalized elliptic genus of $X$. This example is the main motivation of toric forms \cite{BL}.
\end{remark}

\vskip.2cm
Furthermore they proved that the toric forms are modular forms under certain conditions on $\deg$.

\vskip.2cm
\begin{theorem} [Borisov, Gunnells]
	Suppose $\deg$ function takes values in $\frac{1}{l}\mathbb{Z}$, and $\deg$ is not integral valued on the primitive generator of any 1-cone of $\Sigma$. Then toric form  $f_{N,\deg}(q)$ is holomorphic modular form of weight $r$ for the congruence subgroup $\Gamma_1(l)$.
\end{theorem}

Borisov and Gunnells gave a topological interpretation of toric forms by Hirzebruch-Riemann-Roch theorem. Let $\{ d_i\}$ be the set of primitive generator of  1-cone of $\Sigma$, $X$ be toric variety associated to $\Sigma$, and for each $d_i$, $D_i\subset X$ be the corresponding toric divisor. In the following, we abuse $D_i$ to mean either the divisor or its cohomology class. Recall that
$$\theta(z,\tau)=2 q^{1/8}\sin(\pi z)\prod_{j=1}^\infty [(1-q^j)(1-e^{2\pi i  z}q^j)(1-e^{-2\pi i  z}q^j)].$$

In the following paper, we abbreviate $\theta(z,\tau)$ as $\theta(z)$.

\vskip.2cm

\begin{theorem}[ Borisov, Gunnells]\label{BG}
	Assume that the toric variety $X$ is nonsingular, assign to  all primitive generators of  1-cones of $\Sigma$ complex numbers $\alpha_i \notin \mathbb{Z}$. Then
	$$f_{N,\deg}(q)=\int_X \prod_i \frac{(D_i/2\pi i)\theta(D_i/2\pi i-\alpha_i,\tau)\theta'(0,\tau)}{\theta(D_i/2\pi i,\tau)\theta(-\alpha_i,\tau)},$$
where \  $q=\exp^{2\pi i\tau }$.
\end{theorem}

\vskip.2cm
Our ``iterated residue" can be applied to calculate above Euler characteristic for some generalized Bott manifolds, thus getting interesting theta function identities regarding to toric forms.

\subsection{$\mathbb{C}P^2$ bundle over $\mathbb{C}P^2$}\label{classic examples}   \

Consider a $\mathbb{C}P^2$ bundle over $\mathbb{C}P^2$, its corresponding fan is spanned by five primitive vectors $\textbf{e}_1,\ \textbf{e}_2,\ \textbf{e}_3,\ \textbf{e}_4,\ -\textbf{e}_1-\textbf{e}_2+j\textbf{e}_3+k\textbf{e}_4,\ -\textbf{e}_3-\textbf{e}_4$ in $\mathbb{Z}^4$.

\vskip.2cm

Assume that $\deg$ takes $\frac{1}{2}$ on all generators $\textbf{e}_1,\ \textbf{e}_2,\ \textbf{e}_3,\ \textbf{e}_4,\ -\textbf{e}_1-\textbf{e}_2+j\textbf{e}_3+k\textbf{e}_4,\ -\textbf{e}_3-\textbf{e}_4$ in $\mathbb{Z}^4$. Then the toric form is
$$f_{N,\deg}(q)=\sum_{a,b,c,d\in \mathbb{Z}} \frac{2(1+q^{jc+kd})}{(1+q^a)(1+q^b)(1+q^{c})(1+q^{d})(1+q^{jc+kd-a-b})(1+q^{-c-d}) }.$$
On the other hand, by a direct calculation,
	\begin{align*}
		& f_{N,\deg}(q)\\
		= & \int_X \prod_{i=1}^6 \frac{(D_i/2\pi i)\theta(D_i/2\pi i-\alpha_i)\theta'(0)}{\theta(D_i/2\pi i)\theta(-\alpha_i)} \\
		=& \frac{\theta'(0)^6}{(2\pi i)^6  \theta^6(\frac{1}{2})}\textrm{Res}_0 \left\{ \textrm{Res}_0   \frac{\theta^3(\frac{u}{2\pi i}-\frac{1}{2}) \theta(\frac{v }{2\pi i}-\frac{1}{2}) \theta(\frac{v-ju}{2\pi i}-\frac{1}{2})\theta(\frac{v-ku}{2\pi i}-\frac{1}{2})}{\theta^3(\frac{u}{2\pi i})\theta(\frac{v}{2\pi i})\theta(\frac{v-ju}{2\pi i})\theta(\frac{v-ku}{2\pi i})} du\right\}dv\\
		=&  \frac{\theta'(0)^6}{(2\pi i)^4\theta^6(\frac{1}{2})}\textrm{Res}_0 \frac{\theta(v-\frac{1}{2})}{\theta(v)} \left\{ \textrm{Res}_0   \frac{\theta^3(u-\frac{1}{2}) \theta(v-ju-\frac{1}{2})\theta(v-ku-\frac{1}{2})}{\theta^3(u)\theta(v-ju)\theta(v-ku)} du\right\}dv\\
		=&\frac{\theta'(0)^3}{(2\pi i)^4 \theta^3(\frac{1}{2})}\textrm{Res}_0 \frac{\theta^3(v+\frac{1}{2})}{\theta^3(v)}\bigg\{ \frac{3\theta''(\frac{1}{2})}{2\theta(\frac{1}{2})}-\frac{\theta^{(3)}(0)}{2\theta'(0)}+
\frac{j^2+k^2}{2}\Big(\frac{\theta''(v+\frac{1}{2})}{\theta(v+\frac{1}{2})} -\frac{\theta''(v )}{\theta(v )} \Big) \\
		&+jk\Big(\frac{\theta'^2(v+\frac{1}{2})}{\theta^2(v+\frac{1}{2})} -\frac{\theta'^2(v )}{\theta^2(v )} \Big)+(j+k)^2\frac{\theta'^2(v )}{\theta^2(v )}-(j+k)^2\frac{\theta' (v )\theta' (v+\frac{1}{2})}{\theta (v )\theta  (v+\frac{1}{2})}   \bigg\}   dv\\
		=& \frac{1}{(2\pi i)^4}\bigg\{\Big(\frac{3\theta''(\frac{1}{2})}{2\theta(\frac{1}{2})}-
\frac{\theta^{(3)}(0)}{2\theta'(0)}\Big)^2+(j^2+k^2-jk)\Big(\frac{\theta''^2(\frac{1}{2})}{4\theta^2(\frac{1}{2})}
		+\frac{5\theta^{(4)}(\frac{1}{2})}{24\theta(\frac{1}{2})}\\
	&-\frac{7\theta^{(3)}(0)\theta''(\frac{1}{2})}{12\theta'(0)\theta(\frac{1}{2})}-
\frac{\theta^{(5)}(0)}{24\theta'(0)}+\frac{(\theta^{(3)}(0))^2}{6\theta'^2(0)}\Big)\bigg  \}
\end{align*}
so we have that

\vskip .2cm
\begin{proposition}\label{second}
	There exists the following theta function identity
		\begin{align*}
			&\sum_{a,b,c,d\in \mathbb{Z}} \frac{2(1+q^{jc+kd})}{(1+q^a)(1+q^b)(1+q^{c})(1+q^{d})(1+q^{jc+kd-a-b})(1+q^{-c-d}) }\\
			=& \frac{1}{(2\pi i)^4}\Big\{\Big(\frac{3\theta''(\frac{1}{2})}{2\theta(\frac{1}{2})}-
\frac{\theta^{(3)}(0)}{2\theta'(0)}\Big)^2+(j^2+k^2-jk)\Big(\frac{\theta''^2(\frac{1}{2})}{4\theta^2(\frac{1}{2})}
			+\frac{5\theta^{(4)}(\frac{1}{2})}{24\theta(\frac{1}{2})}-\frac{7\theta^{(3)}(0)\theta''(\frac{1}{2})}{12\theta'(0)\theta(\frac{1}{2})}\\
			&-\frac{\theta^{(5)}(0)}{24\theta'(0)}+\frac{(\theta^{(3)}(0))^2}{6\theta'^2(0)}\Big)  \Big\}.
	\end{align*}
\end{proposition}

\vskip.2cm

\subsection{Hirzebruch surface}

Hirzebruch surface is a $\mathbb{C} P^1$  bundle over $\mathbb{C} P^1$, denoted by  $F_k=\mathbb{C} P(\underline{\mathbb{C}}\oplus\mathcal{O}(k))$ .  Its corresponding fan  is spanned by four primitive vectors  $\textbf{e}_1, \ \textbf{e}_2, \  -\textbf{e}_1+k\textbf{e}_2, \ -\textbf{e}_2$, see Figure $2$.
\begin{center}
	\begin{tikzpicture}
		\fill (0,0)  circle (2pt);
		\draw[->] (0,0)--(2,0);
		\node[right] at (2,0) {$\textbf{e}_1$};
		
		\draw[->] (0,0)--(0,2);
		\node[right] at (0,2) {$\textbf{e}_2$};
		
		\draw[->] (0,0)--(0,-2);
		
		\node[right] at (0,-2) {$-\textbf{e}_2$};
		
		\draw[->] (0,0)--(-1,2);
		
		\node[left] at (-1,2) {$-\textbf{e}_1+k\textbf{e}_2$};
		
		\node[below] at (0,-2.5) {Figure 2};
	\end{tikzpicture}
\end{center}

Assume that $\deg$ takes $\alpha_1, \alpha_2, \alpha_3, \alpha_4$ on the generators $ \textbf{e}_1,\ -\textbf{e}_1+k\textbf{e}_2,\ -\textbf{e}_2,\ \textbf{e}_2,$ respectively. Then the toric form is
\begin{align*}
	& f_{N,\deg}(q)\\= & \sum_{a,b\in \mathbb{Z}} \frac{1}{(1-e^{2\pi i \alpha_1}q^b)(1-e^{2\pi i \alpha_4}q^a)}+ \frac{1}{(1-e^{2\pi i \alpha_1}q^b)(1-e^{2\pi i \alpha_3}q^{-a})}\\
	& + \frac{1}{(1-e^{2\pi i \alpha_3}q^{-a})(1-e^{2\pi i \alpha_2}q^{ka-b})}+ \frac{1}{(1-e^{2\pi i \alpha_4}q^{a})(1-e^{2\pi i \alpha_2}q^{ka-b})}\\
	& -\frac{1}{1-e^{2\pi i \alpha_4}q^a}-\frac{1}{1-e^{2\pi i \alpha_1}q^b}-\frac{1}{1-e^{2\pi i \alpha_2}q^{ka-b}}-\frac{1}{1-e^{2\pi i \alpha_3}q^{-a}}+1\\
	=&\sum_{a,b\in \mathbb{Z}} \frac{(1-e^{2\pi i (\alpha_3+\alpha_4)})(1-e^{2\pi i (\alpha_1+\alpha_2)}q^{ka})}{(1-e^{2\pi i \alpha_4}q^a)(1-e^{2\pi i \alpha_1}q^b)(1-e^{2\pi i \alpha_3}q^{-a})(1-e^{2\pi i \alpha_2}q^{ka-b}) }.
\end{align*}
On the other hand, we have
\begin{align*}
	& f_{N,\deg}(q)\\
	= & \int_X \prod_{i=1}^4 \frac{(D_i/2\pi i)\theta(D_i/2\pi i-\alpha_i)\theta'(0)}{\theta(D_i/2\pi i)\theta(-\alpha_i)} \\
	=& \frac{\theta'(0)^4}{(2\pi i)^4\prod_{i=1}^4 \theta(-\alpha_i)}\textrm{Res}_0 \left\{ \textrm{Res}_0   \frac{\theta(\frac{u}{2\pi i}-\alpha_1)\theta(\frac{u}{2\pi i}-\alpha_2)\theta(\frac{v}{2\pi i}-\alpha_3)\theta(\frac{v-ku}{2\pi i}-\alpha_4)}{\theta^2(\frac{u}{2\pi i})\theta(\frac{v}{2\pi i})\theta(\frac{v-ku}{2\pi i})} du\right\}dv\\
	=&  \frac{\theta'(0)^4}{(2\pi i)^2\prod_{i=1}^4 \theta(-\alpha_i)}\textrm{Res}_0 \frac{\theta(v-\alpha_3)}{\theta(v)} \left\{ \textrm{Res}_0   \frac{\theta(u-\alpha_1)\theta(u-\alpha_2)\theta(v-ku-\alpha_4)}{\theta^2(u)\theta(v-ku)} du\right\}dv\\
	=&\frac{\theta'(0)^2}{(2\pi i)^2  \theta(-\alpha_3)\theta(-\alpha_4)}\textrm{Res}_0 \frac{\theta(v-\alpha_3)\theta(v-\alpha_4)}{\theta^2(v)}\Big\{\frac{\theta'(-\alpha_1)}{\theta(-\alpha_1)}+ \frac{\theta'(-\alpha_2)}{\theta(-\alpha_2)}-k\frac{\theta'(v-\alpha_4)}{\theta(v-\alpha_4)}\\
& +k\frac{\theta'(v)}{\theta(v)} \Big\} dv.
\end{align*}
Similarly, we have
$$\textrm{Res}_0\ \frac{\theta(v-\alpha_3)\theta(v-\alpha_4)}{\theta^2(v)} dv= \frac{\theta(-\alpha_3)\theta(-\alpha_4)}{ \theta'^2(0)}\left\{\frac{\theta'(-\alpha_3)}{\theta(-\alpha_3)}+ \frac{\theta'(-\alpha_4)}{\theta(-\alpha_4)}     \right\},$$
$$\textrm{Res}_0\ \frac{\theta(v-\alpha_3)\theta'(v-\alpha_4)}{\theta^2(v)} dv= \frac{\theta(-\alpha_3)\theta'(-\alpha_4)}{ \theta'^2(0)}\left\{\frac{\theta'(-\alpha_3)}{\theta(-\alpha_3)}+ \frac{\theta''(-\alpha_4)}{\theta'(-\alpha_4)}   \right  \}.$$
and
\begin{align*}
	& \textrm{Res}_0\ \frac{\theta(v-\alpha_3)\theta(v-\alpha_4)\theta'(v)}{\theta^3(v)}dv\\
	=& \textrm{Res}_0\ \frac{(\theta(- \alpha_3)+\theta'(-\alpha_3)v+\frac{\theta''(-\alpha_3)v^2}{2})(\theta(- \alpha_4)+\theta'(-\alpha_4)v+\frac{\theta''(-\alpha_4)v^2}{2})  (\theta'(0)+\frac{\theta^{(3)}(0)v^2}{2})}{(\theta'(0)v+\theta^{(3)}(0)v^3/6)^3}dv\\
	=& \textrm{Res}_0\ \frac{(\theta(- \alpha_3)+\theta'(-\alpha_3)v+\frac{\theta''(-\alpha_3)v^2}{2})(\theta(- \alpha_4)+\theta'(-\alpha_4)v+\frac{\theta''(-\alpha_4)v^2}{2})} { \theta'^2(0) v^3 }dv \\
	=&\frac{\theta(- \alpha_3)\theta(- \alpha_4)}{2\theta'^2(0)}\left\{ ( \frac{\theta''(-\alpha_3)}{\theta(-\alpha_3)}+\frac{\theta''(-\alpha_4)}{\theta(-\alpha_4)}+2\frac{\theta'(-\alpha_3)\theta'(-\alpha_4)}{\theta(-\alpha_3)\theta(-\alpha_4)}) \right\}.
\end{align*}

Combining with the above calculations, we have that

\vskip .2cm

\begin{proposition}\label{Hir} For Hirzebruch surface $F_k=\mathbb{C} P(\underline{\mathbb{C}}\oplus\mathcal{O}(k))$, we have  theta function identity
	\begin{align*}
		&\sum_{a,b\in \mathbb{Z}} \frac{(1-e^{2\pi i (\alpha_3+\alpha_4)})(1-e^{2\pi i (\alpha_1+\alpha_2)}q^{ka})}{(1-e^{2\pi i \alpha_4}q^a)(1-e^{2\pi i \alpha_1}q^b)(1-e^{2\pi i \alpha_3}q^{-a})(1-e^{2\pi i \alpha_2}q^{ka-b}) }\\
		=&\frac{1}{(2\pi i)^2} \left\{(\frac{\theta'(-\alpha_3)}{\theta(-\alpha_3)}+ \frac{\theta'(-\alpha_4)}{\theta(-\alpha_4)} )(\frac{\theta'(-\alpha_1)}{\theta(-\alpha_1)}+ \frac{\theta'(-\alpha_2)}{\theta(-\alpha_2)})+\frac{k}{2}(\frac{\theta''(-\alpha_3)}{\theta(-\alpha_3)}- \frac{\theta''(-\alpha_4)}{\theta(-\alpha_4)} )   \right\}.
\end{align*}\end{proposition}
\vskip.2cm
It is easy to check that if $\alpha_3+\alpha_4\in \mathbb{Z}$, both sides equal zero.

\subsection{Vanishing result}\

Besides the accurate calculation in low dimensions, we can also get a Landweber-Stong type vanishing result for a family of generalized Bott manifolds.

\vskip.2cm

\begin{theorem}\label{vanish toric}
	Let $V$ be the toric variety $\CC P(\eta^{\otimes i_1}\oplus \eta^{\otimes i_2}\oplus \eta^{\otimes i_3}\oplus \underline{\mathbb{C}})$ over $\CC P^{n_1}$, where $\eta$ denotes the tautological bundle of $\CC P^{n_1}$. 	
	If $(i_1, \ i_2,\  i_3)$ are coprime such that $ \ (i_1,\ i_1-i_2,\ i_1-i_3), \ (i_2,\ i_2-i_1,\ i_2-i_3), \ (i_3,\ i_3-i_1,\ i_3-i_2)$ are also coprime respectively, $\sum_{j=0}^{3} \alpha_{n_1+2+j}\in \mathbb{Z}$, and $\sum_{i=1}^{n_1+1} \alpha_i- \sum_{j=1}^{3}i_j \alpha_{n_1+2+j}\in \mathbb{Z}$, then the toric form $f_{N,\deg}(q)=0$.
\end{theorem}

\vskip.2cm
The proof of Theorem~ \ref{vanish toric} is similar to that of Theorem~\ref{technique}.

\section{Witten genus: a vanishing result}\label{Witten genus}

Let $M$ be a $4k$ dimensional closed oriented smooth manifold, $E$ be a complex vector bundle over $M$. For any complex number $t$, set
$$S_t(E)=\underline{\CC}+tE+t^2S^2(E)+\cdots ,$$
where $S^j(E)$ is the $j$-th symmetric power of $E$.

\vskip.2cm

Let $q=e^{2\pi i\tau}$ with $\tau\in \mathbb{H}$, the upper half plane. Witten defined
$$\Theta_q(E)=\underset{n\geq 1}{\bigotimes}\ S_{q^n}(E),$$ and then  the \textbf{Witten genus}
$$\varphi_W(M)=\left\langle\widehat{A}(M)ch(\Theta_q(TM\otimes \CC-\underline{\CC}^{4k})), [M]\right \rangle.$$

\vskip.2cm
Let ${\pm 2\pi i x_j \ ( 1\leq j\leq 2k)}$ be the formal Chern roots of $TM\otimes \CC$. Then the Witten genus can be rewritten as (c.f. \cite{Liu})

$$\varphi_W(M)=\left\langle \left( \prod_{j=1}^{2m}x_{j}\frac{\theta ^{\prime }(0,\tau
	)}{\theta (x_{j},\tau )}\right) ,[M]\right\rangle,$$
where $\theta(z,\tau)$ is the Jacobi theta function
$$\theta(z,\tau)=2 q^{1/8}sin(\pi z)\prod_{j=1}^\infty [(1-q^j)(1-e^{2\pi i z}q^j)(1-e^{-2\pi i z}q^j)]$$ where $q=e^{2\pi i \tau}.$


\vskip.2cm
The manifold $M$ is called spin if $\omega_1(M)=0,\ \omega_2(M)=0,$ where $\omega_1(M),\ \omega_2(M)$ are the first and second Stiefel-Whitney classes of $M$. According to Atiyah-Singer index theorem, when $M$ is spin, $\varphi_W(M)\in \Z[[q]]$.

\vskip.2cm

A spin manifold $M$ is called string if the characteristic class $\frac{1}{2}p_1(M)=0$, where $p_1(M)$ is the first Pontryagin class of $M$. It is well known that if $M$ is string, $\varphi_W(M)$ is a modular form of weight $2k$ over $SL(2,\mathbb{Z})$, see \cite[Hirzebruch]{HBJ94}.  In this section, we mainly discuss the Witten genus of string manifolds.

\subsection{String complete intersections in  generalized Bott manifolds}\

Consider a two staged generalized Bott manifold $V \coloneqq \CC P(\eta^{\otimes i_1}\oplus \cdots \eta^{\otimes i_{n_2}}\oplus \underline{\CC}) $ over $\CC P^{n_1}$, where $\eta$ denotes the tautological bundle of $\CC P^{n_1}$.

\vskip.2cm

Let $\textbf{I}=(i_1,\cdots,i_{n_2})$ denote the index of projective bundle and $H^\textbf{I}_{n_1,n_2}(d_1,d_2)$  denote the submanifold Poincar\'e dual to $d_1u+d_2v\in H^2(V;\Z)$. Let  $i: H^\textbf{I}_{n_1,n_2}(d_1,d_2)\longrightarrow V$ be the natural embedding, and $\nu$ denote the normal bundle of this embedding. We have
$$i^*(TV)\cong i^*(\nu)\oplus TH^\textbf{I}_{n_1, n_2}(d_1,d_2).$$
Thus$$c_1(H^\textbf{I}_{n_1,n_2}(d_1,d_2))=i^*((n_1+1)u+v+\sum_{j=1}^{n_2} (v-i_j u)- (d_1u+d_2v) ).$$ and
 \begin{align*}
 p_1(H^{\bf I}_{n_1,n_2}(d_1,d_2))= &\ i^*\Big\{\big(n_1+1+\sum_{j=1}^{n_2}i_j^2-d_1^2\big)u^2+
 (1+n_2-d_2^2)v^2\\
 &-2\big(d_1d_2+\sum_{j=1}^{n_2}i_j\big)uv\Big\}.
 \end{align*}


By Lefschetz hyperplane theorem,   it is obvious that there is no torsion in the cohomology ring of complete intersections in generalized Bott manifolds, thus $p_1=0$ implies $\frac{1}{2}p_1=0$.
\vskip.2cm
\begin{lemma}\label{string}
	Twisted Milnor hypersurfaces
	$H^{\bf I}_{n_1,n_2}(d_1,d_2)$ can not be string for $n_2\geq3$.
\end{lemma}

\begin{proof}
	Let $i_!: H^*(H^{\bf I}_{n_1,n_2}(d_1,d_2)) \longrightarrow H^*(V)$ be the pushforward map.
		First we assume
	$$		\begin{cases}
			n_1+1+\sum\limits_{j=1}^{n_2}i_j^2=d_1^2\\
			1+n_2=d_2^2\\
			d_1d_2+\sum_{j=1}^{n_2}i_j=0.
		\end{cases}
$$
\vskip .2cm
	\noindent Thus $  ( n_1+1+\sum_{j=1}^{n_2}i_j^2)(1+n_2)=(d_1d_2)^2=(\sum_{j=1}^{n_2}i_j)^2$,
	which is impossible, and so  $(n_1+1+\sum_{j=1}^{n_2}i_j^2-d_1^2)u^2+(1+n_2-d_2^2)v^2+(-2\sigma_1-2d_1d_2)uv\neq 0$.
	
	\vskip.2cm
	On the other hand, if $p_1(H^{\bf I}_{n_1,n_2}(d_1,d_2))=0$, then
	\begin{align*}
		&i_!p_1(H^{\bf I}_{n_1,n_2}(d_1,d_2))\\
		=&i_!i^*\left\{(n_1+1+\sum_{j=1}^{n_2}i_j^2-d_1^2)u^2+(1+n_2-d_2^2)v^2+(-2\sigma_1-2d_1d_2)uv\right\}\\
		=&(d_1u+d_2v)\left\{(n_1+1+\sum_{j=1}^{n_2}i_j^2-d_1^2)u^2+(1+n_2-d_2^2)v^2+(-2\sigma_1-2d_1d_2)uv\right\}\\
		=&0.
	\end{align*}
	Since $d_1u+d_2v\neq 0$ as long as $ n_2\geq3$, it is impossible for $p_1=0$.
\end{proof}

\vskip .2cm
\begin{remark}
	It is possible for twisted Milnor hypersurface $H^{\bf I}_{n_1, n_2}(d_1, d_2)$ to be string when $n_2<3$. For example, choose $n_2=1, d_2=0$, we have
	$$ p_1=i^*\bigg\{(n_1+1+i_1^2-d_1^2)u^2+2v^2-2i_1uv\bigg\},$$
	where $2v^2-2i_1uv$ is killed by  the relation in the cohomology ring of 2 staged generalized Bott manifold. Thus if $n_1+1+i_1^2-d_1^2=0$, we have $p_1=0$. For example,   twisted Milnor hypersurface $H^{\pm 6}_{12, 1}(\pm 7, 0)$ is the required 24-dim string manifold.
\end{remark}

\vskip .2cm
Due to the restriction of Lemma \ref{string}, we proceed to investigate string complete intersections. The submanifold $H^{\bf I}_{n_1, n_2}(d_1, d_2; d_3, d_4)$ Poincar\'e dual to cohomology class $(d_1u+d_2v)\cdot(d_3u+d_4v)\in H^4(V;\Z)$  is string as long as the equation system
$${\centering
	\left\{
	\begin{array}{lr}
		n_1+1+\sum\limits_{j=1}^{n_2}i_j^2=d_1^2+d_3^2\\
		1+n_2=d_2^2+d_4^2\\
		d_1d_2+d_3d_4+\sum\limits_{j=1}^{n_2}i_j=0
	\end{array}\right.}$$
has integer solutions. Actually, there exist plenty of string complete intersections.

\subsection{Witten genus}\
Consider the Jacobi theta function
$$\theta(z,\tau)=2 q^{1/8}sin(\pi z)\prod_{j=1}^\infty [(1-q^j)(1-e^{2\pi i  z}q^j)(1-e^{-2\pi i  z}q^j)],$$
where $\theta$ admits some periodicity such that $\theta(z+1,\tau)=-\theta(z,\tau)$ and $\theta(z+b\tau,\tau)=(-1)^be^{-\pi i (2b z+b^2\tau)}\theta(z,\tau)$. Clearly  $\theta$ has simple zeros on lattice $\Z+\Z\tau$ and no pole. Applying our iterated residue, the Witten genus can be reformulated as
\begin{align*}
	&\varphi_W(H^{\bf I}_{n_1,n_2})\\
=&\left\langle\widehat{A}(H^{\bf I}_{n_1,n_2})ch(\Theta_q(TH^{\bf I}_{n_1,n_2}\otimes \CC-\underline{\CC}^{4k})), [H^{\bf I}_{n_1,n_2}]\right\rangle\\
=& \theta'^{n_1+n_2}(0)\Big\langle (\frac{u}{\theta(u)})^{n_1+1} \frac{v}{\theta(v)}\prod_{j=1}^{n_2} \frac{v-i_ju}{\theta(v-i_ju)}\frac{\theta(d_1u+d_2v)}{d_1u+d_2v}\frac{\theta(d_3u+d_4v)}{d_3u+d_4v}, [H^{\bf I}_{n_1,n_2}]  \Big\rangle\\
	=& \frac{\theta'^{n_1+n_2}(0)}{(2\pi i)^{ -2}} \Big\langle (\frac{u}{\theta(u)})^{n_1+1} \frac{v}{\theta(v)}\prod_{j=1}^{n_2} \frac{v-i_ju}{\theta(v-i_ju)}\theta(d_1u+d_2v)\cdot \theta(d_3u+d_4v), [V]  \Big\rangle\\
	=&  \frac{\theta'^{n_1+n_2}(0)}{(2\pi i)^{n_1+n_2-2}} \mathrm{Res}_0\left\{ \mathrm{Res}_0 \frac{\theta(d_1u+d_2v)\cdot \theta(d_3u+d_4v)}{\theta^{n_1+1}(u)\theta(v)\prod_{j=1}^{n_2} \theta(v-i_ju)}du\right\}dv.
\end{align*}
\begin{remark}
	If meromorphic function $g$ has pole of order $n$ at point $c$, then $$\mathrm{Res}_c\, g=\frac{1}{(n-1)!}\cdot \underset{z\rightarrow c}{\lim}\ \frac{d^{n-1}}{dz^{n-1}}((z-c)^ng(z)).$$ It is usually difficult to calculate the residue at high order poles for both variables $u$ and $v$ directly, thus we shall make use of residue theorem to reduce   iterated residue at high order poles to simple poles. This is the common operation for Landweber-Stong type vanishing results.
\end{remark}

\vskip .2cm
Under certain condition, we obtain a Landweber-Stong type vanishing result for string complete intersections in generalized Bott manifolds. We add the coprime condition to make sure that the poles except for $0$ are simple.
\begin{theorem}\label{technique}
	For a string complete intersection $H^{\bf I}_{n_1, \ 3}(d_1, d_2; d_3,d_4)$, if ${\bf I}=(i, \ j,\  k)$ are coprime such that $ \ (i,\ i-j,\ i-k), \ (j,\ j-i,\ j-k), \ (k,\ k-i,\ k-j)$ are also coprime respectively, then the Witten genus
	$$\varphi_W(H^{\bf I}_{n_1,\ 3}(d_1, d_2; d_3,d_4))=0.$$
\end{theorem}

\begin{proof}
	$$\varphi_W(H^{\bf I})= \frac{\theta'^{n_1+n_2}(0)}{(2\pi i)^{n_1+n_2-2}} \mathrm{Res}_0 \left\{\mathrm{Res}_0 \frac{\theta(d_1u+d_2v) \theta(d_3u+d_4v)}{\theta^{n_1+1}(u)\theta(v)\theta(v+iu)\theta(v+ju)\theta(v+ku)}du\right\}dv.$$
	
	Denote $\omega(u)=\frac{\theta(d_1u+d_2v)\cdot \theta(d_3u+d_4v)}{\theta^{n_1+1}(u)\theta(v)\prod_{j=1}^{n_2} \theta(v-i_ju)}du$, we check its periodicity
	$$\omega(u+1)=(-1)^{d_1+d_3-n_1-1-\sum i_j}\omega(u)=\omega(u);$$
	\begin{align*}&\omega(u+ \tau)\\
= &\ \frac{(-1)^{d_1+d_3}e^{-\pi i(2d_1(d_1u+d_2v)+2d_3(d_3u+d_4v)+(d_1^2+d_3^2)\tau)}\theta(d_1u+d_2v)\cdot \theta(d_3u+d_4v)}{(-1)^{n_1+1+\sum i_j}e^{-\pi i((n_1+1)(2u+\tau)+\sum (-2i_j(v-i_ju)+i_j^2 \tau))} \theta^{n_1+1}(u)\theta(v)\prod \theta(v-i_ju)}du.
\end{align*}
	\vskip.2cm
	Since $H^{\bf I}_{n_1,n_2}(d_1,d_2; d_3,d_4)$ is string, we have $\omega(u+ \tau)=\omega(u)$. We can claim that $\omega(u)$ admits double periodicity on  lattice $L\coloneqq  \Z+\Z\tau $. Let $T^2\cong \CC/L$, thus $\omega(u)$ can be defined on closed surface $T^2$. Then we can apply residue theorem on $\omega(u)$.
	Since $\theta$ has simple zeros on $ \Z+\Z\tau $ and no poles,  the possible poles of $\omega$ are
	$$0,\  \frac{a+b\tau-v}{i} (1\leq a,\ b\leq i),\ \frac{c+d\tau-v}{j}(1\leq c,\ d\leq j),\ \frac{m+n\tau-v}{k}(1\leq m,\ n\leq k).$$
	
	Applying residue theorem on $u$, we get
	$$\{\mathrm{Res}_0+\mathrm{Res}_{1+\tau-\frac{v}{i}}+\mathrm{Res}_{1+\tau-\frac{v}{j}}+\mathrm{Res}_{1+\tau-\frac{v}{k}}+\mathrm{Res}_{\frac{a+b\tau-v}{i}}+\mathrm{Res}_{\frac{c+d\tau-v}{j}}+\mathrm{Res}_{\frac{m+n\tau-v}{k}}  \} \omega(u)=0. $$
Moreover
	
\begin{align*} &\mathrm{Res}_0\{\mathrm{Res}_{1+\tau-\frac{v}{i}}\ \omega(u)\}dv=\frac{1}{\theta'(0)}\cdot\mathrm{Res}_0 \frac{\theta((d_2i-d_1)v)\theta((d_4i-d_3)v)}{\theta^{n_1+1}
(-v)\theta(iv)\theta((i-j)v)\theta((i-k)v)}dv,\\
	& \mathrm{Res}_0\{\mathrm{Res}_{1+\tau-\frac{v}{j}}\ \omega(u)\}dv=\frac{1}{\theta'(0)} \cdot\mathrm{Res}_0 \frac{\theta((d_2j-d_1)v)\theta((d_4j-d_3)v)}{\theta^{n_1+1}(-v)\theta(jv)\theta((j-i)v)
\theta((j-k)v)}dv,\\
&\mathrm{Res}_0\{\mathrm{Res}_{1+\tau-\frac{v}{k}}\ \omega(u)\}dv=\frac{1}{\theta'(0)} \cdot\mathrm{Res}_0 \frac{\theta((d_2k-d_1)v)\theta((d_4k-d_3)v)}{\theta^{n_1+1}(-v)
\theta(kv)\theta((k-i)v)\theta((k-j)v)}dv,\\
&\mathrm{Res}_0\{\mathrm{Res}_{\frac{a+b\tau-v}{i}}\ \omega(u)\}dv=\frac{(-1)^{a+b+1} }{i\cdot\theta'^2(0)}e^{2\pi i b^2 \tau}\cdot \frac{\theta(\frac{d_1(a+b\tau)}{i})\theta(\frac{d_3(a+b\tau)}{i})}{\theta^{n_1+1}
(\frac{a+b\tau}{i})\theta(\frac{j(a+b\tau)}{i})\theta(\frac{k(a+b\tau)}{i})},\\
&\mathrm{Res}_0\{\mathrm{Res}_{\frac{c+d\tau-v}{j}}\ \omega(u)\}dv=\frac{(-1)^{c+d+1} }{j\cdot\theta'^2(0)}e^{2\pi i d^2 \tau}\cdot \frac{\theta(\frac{d_1(c+d\tau)}{j})\theta(\frac{d_3(c+d\tau)}{j})}{\theta^{n_1+1}
(\frac{c+d\tau}{j})\theta(\frac{i(c+d\tau)}{j})\theta(\frac{k(c+d\tau)}{j})},\\
&\mathrm{Res}_0\{\mathrm{Res}_{\frac{m+n\tau-v}{k}}\ \omega(u)\}dv=\frac{(-1)^{m+n+1} }{k\cdot\theta'^2(0)}e^{2\pi i n^2 \tau}\cdot \frac{\theta(\frac{d_1(m+n\tau)}{k})\theta(\frac{d_3(m+n\tau)}{k})}{\theta^{n_1+1}
(\frac{m+n\tau}{k})\theta(\frac{i(m+n\tau)}{k})\theta(\frac{j(m+n\tau)}{k})}.
\end{align*}
	
Apply residue theorem again on $v$ to the followings
	$$\mathrm{Res}_0\{ \mathrm{Res}_{1+\tau-\frac{v}{i}}\ \omega(u)\}dv, \ \mathrm{Res}_0\{ \mathrm{Res}_{1+\tau-\frac{v}{j}}\ \omega(u) \}dv, \ \mathrm{Res}_0\{ \mathrm{Res}_{1+\tau-\frac{v}{k}}\ \omega(u) \}dv. $$
	
	First let us consider $\mathrm{Res}_0\{ \mathrm{Res}_{1+\tau-\frac{v}{i}}\ \omega(u)\}dv$. Denote $$\omega(v)=\frac{\theta((d_2i-d_1)v)\theta((d_4i-d_3)v)}{\theta^{n_1+1}(-v)\theta(iv)\theta((i-j)v)\theta((i-k)v)}dv.$$ Since $i, (i-j), i-k$ are coprime, it is easy to check $\omega(v)$ can also be defined on $T^2\cong \mathbb{C}/L$. The possible poles of $\omega(v)$ include $0, \frac{a+b\tau}{i}, \ \frac{a+b\tau}{i-j}, \ \frac{a+b\tau}{i-k}$. We have
	
	$$\mathrm{Res}_{\frac{a+b\tau}{i}}\ \omega(v)=\frac{(-1)^{a+b} }{i\cdot \theta' (0)}e^{2\pi i b^2 \tau}\cdot \frac{\theta(\frac{(d_2i-d_1)(a+b\tau)}{i})\theta(\frac{(d_4i-d_3)(a+b\tau)}{i})}{\theta^{n_1+1}(-\frac{a+b\tau}{i})\theta(\frac{(i-j)(a+b\tau)}{i})\theta(\frac{(i-k)(a+b\tau)}{i})}.$$
	Since
	$\theta(z+b\tau)=(-1)^b e^{-2\pi i b z-\pi i b^2 \tau}\theta(z)$,
	
	\begin{align*}
		&\frac{\mathrm{Res}_{\frac{a+b\tau}{i}}\ \omega(v)}{\mathrm{Res}_0\{\mathrm{Res}_{\frac{a+b\tau-v}{i}}\ \omega(u)\}dv}\\=&  (-1)^{n_1+1+(d_2+d_4)(a+b)}\cdot e^{-2\pi i (-d_2d_1-d_4d_3+j+k)b(\frac{a+b\tau}{i})-\pi i b^2(d_2^2+d_4^2-2) \tau} \\
		=& 1.
	\end{align*}
	
	Similarly, we have that
	$$\mathrm{Res}_{\frac{c+d\tau}{j}}\{\mathrm{Res}_{1+\tau-\frac{v}{j}}\ \omega(u)\}dv=\mathrm{Res}_0\{\mathrm{Res}_{\frac{c+d\tau-v}{j}}\ \omega(u)\}dv,$$
	
	$$\mathrm{Res}_{\frac{m+n\tau}{k}}\ \{\mathrm{Res}_{1+\tau-\frac{v}{k}}\ \omega(u)\}dv=\mathrm{Res}_0\{\mathrm{Res}_{\frac{m+n\tau-v}{k}}\ \omega(u)\}dv$$
	so
	\begin{align*}
		&\frac{\mathrm{Res}_{\frac{a+b\tau}{i-j}}\ \omega(v)}{\mathrm{Res}_{\frac{a+b\tau}{i-j}}\{\mathrm{Res}_{1+\tau-\frac{v}{j}}\ \omega(u)\}dv}\\=&  (-1)^{1+(d_2+d_4)(a+b)}\cdot e^{-2\pi i (d_2(d_2j-d_1)+d_4(d_4j-d_3)-2j+k)b\frac{a+b\tau}{i-j}-\pi i b^2(d_2^2+d_4^2-2) \tau} \\
		=& -1.
	\end{align*}
Then	
	$$\mathrm{Res}_{\frac{a+b\tau}{i-k}}\ \omega(v)=-\mathrm{Res}_{\frac{a+b\tau}{i-k}}\{\mathrm{Res}_{1+\tau-\frac{v}{k}}\ \omega(u)\}dv$$
	and
	$$\mathrm{Res}_{\frac{a+b\tau}{j-k}}\{\mathrm{Res}_{1+\tau-\frac{v}{j}}\ \omega(u)\}dv=-\mathrm{Res}_{\frac{a+b\tau}{j-k}}\{\mathrm{Res}_{1+\tau-\frac{v}{k}}\ \omega(u)\}dv.$$
	
	\vskip.2cm
	With all the calculations together, we have $\mathrm{Res}_0\{ \mathrm{Res}_0\ \omega(u)\}dv=0$, i.e.,
	$$\varphi_W(H^{\bf I}_{n_1,\ 3}(d_1, d_2; d_3,d_4))=0.$$
\end{proof}


\section{Statements and Declarations}
No conflict of interest exits in the submission of this manuscript, and manuscript is approved by all authors for publication. I would like to declare on behalf of my co-authors that the work described was original research that has not been published previously, and not under consideration for publication elsewhere, in whole or in part. All the authors listed have approved the manuscript that is enclosed.

\section{Data availability statement}
Data sharing is not applicable to this article as no new data were created or analyzed in this study.

\end{document}